\documentclass[12pt]{amsart}

\usepackage[unicode]{hyperref}
\usepackage{color}

\usepackage{amsmath,amsfonts,amssymb,amsthm}
\usepackage{verbatim}

\usepackage{longtable}

\numberwithin{equation}{section}

\newtheorem{dummy}{dummy}[section]

\newtheorem{theorem}[dummy]{Theorem}
\newtheorem*{theorem*}{Theorem}
\newtheorem{proposition}[dummy]{Proposition}
\newtheorem*{proposition*}{Proposition}

\newtheorem{lemma}[dummy]{Lemma}
\newtheorem{corollary}[dummy]{Corollary}
\newtheorem*{corollary*}{Corollary}
\newtheorem{conjecture}[dummy]{Conjecture}

\theoremstyle{definition}

\newtheorem{definition}[dummy]{Definition}
\newtheorem{remark}[dummy]{Remark}
\makeatletter
\newcommand{\symbitem}[1]{\item[#1]%
\renewcommand{\@currentlabel}{#1}\ignorespaces}
\makeatother
\newcommand{\beq}{\begin{equation}}
\newcommand{\eeq}{\end{equation}}
\newcommand{\beqa}{\begin{eqnarray}}
\newcommand{\eeqa}{\end{eqnarray}}
\newcommand{\beaa}{\begin{eqnarray*}}
\newcommand{\ben}{\begin{eqnarray*}}
\newcommand{\eaa}{\end{eqnarray*}}
\newcommand{\een}{\end{eqnarray*}}
\def \Db {\D^b}

\def \AA {\mathcal{A}}
\def \BB {\mathcal{B}}
\def \CC {\mathcal{C}}
\def \D {\mathcal{D}}

\def \FF {\mathcal{F}}

\def \GG {\mathcal{G}}

\def \O {\mathcal{O}}
\def\OO {\mathcal{O}}

\def \T {\mathcal{T}}

\def \C {\mathbb{C}}
\def \F {\mathbb{F}}

\def \L {\mathbb{L}}
\def \P {\mathbb{P}}
\def \Q {\mathbb{Q}}
\def \R {\mathbb{R}}

\def \Z {\mathbb{Z}}

\def \Aut {\mathrm{Aut}}

\def \ge {\geqslant}
\def \geq {\geqslant}
\def \le {\leqslant}
\def \leq {\leqslant}

\def \kappa {\varkappa}

\def\={\;=\;}
\def\bal{\begin{aligned}}
\def\eal{\end{aligned}}

\newcommand{\udot}{{\:\raisebox{3pt}{\text{\circle*{1.5}}}}}
\def \bullet {\udot}

\DeclareMathOperator{\Hom}{Hom}

\DeclareMathOperator{\rk}{rk}

\DeclareMathOperator{\Pic}{Pic}
\providecommand{\arxiv}[1]{\href{http://arxiv.org/abs/#1}{arXiv:#1}}

\date{\today}
\author{Sergey Galkin, Ludmil Katzarkov, Anton Mellit, Evgeny Shinder}

\thanks{This work was partially supported by World Premier
 International Research Center Initiative (WPI Initiative), MEXT,
  Japan, Grant-in-Aid for Scientific Research (10554503) from Japan
  Society for Promotion of Science, and AG Laboratory NRU-HSE, RF
  government grant, ag. 11.G34.31.0023.  
S.\,G. and L.\,K. were funded
  by grants NSF DMS0600800, NSF FRG DMS-0652633, NSF FRG DMS-0854977,
  NSF DMS-0854977, NSF DMS-0901330, grants FWF P 24572-N25 and FWF
  P20778, and an ERC grant --- GEMIS.
E.\,S. has been supported by the
  Max-Planck-Institut f\"ur Mathematik, SFB 45 Bonn-Essen-Mainz grant
  and the Hausdorff Center, Bonn.  
S.\,G. and A.\,M. would like to
  thank ICTP for inviting them to the ``School and Conference on
  Modular Forms and Mock Modular Forms and their Applications in
  Arithmetic, Geometry and Physics'' to Trieste in March 2011, where
  this collaboration started, and IPMU for A.\,M.'s visit in November
  2011, when the first version of this paper was completed.
Part of this work was done in June 2013 during the visit of S.\,G. 
 to Johannes Gutenberg Universit\"at Mainz, funded by SFB 45.
}

\title{Minifolds and Phantoms}

\begin{document}
\begin{flushleft}
\href{http://db.ipmu.jp/ipmu/sysimg/ipmu/1113.pdf}{IPMU 13-0102}
\end{flushleft}
\begin{abstract}
A minifold is a smooth projective $n$-dimensional variety %$X$ 
such that its bounded derived category of coherent sheaves %$\D^b(X)$ 
admits a semi-orthogonal decomposition into
a collection of $n+1$ exceptional objects. In this paper we
classify minifolds of dimension $n \leq 4$.

We conjecture that the derived category of fake projective spaces have a similar
semi-orthogonal decomposition into
a collection of $n+1$ exceptional objects and a category  
with vanishing Hochschild homology.
We prove this for fake projective planes with non-abelian automorphism group (such as Keum's surface).
Then by passing to equivariant categories
we construct new examples of phantom categories with both Hochschild homology and Grothendieck group vanishing.
%by taking the proper equivariant categories.
%provide analogues of minifolds in the class of 
%are varieties of general type.
%We investigate the structure of the derived category of 
%fake projective planes and find phantom subcategories in some of them.
\end{abstract}
%\begin{center}
\maketitle
%\end{center}
%\tableofcontents
%\newpage
%%%%%%%%%%%%%%%%%%%%%%%%%%%%%%%%%%%%%%%%%%%%%%%%%%%%%%%%%%%%%%%%%%%%%%%%%%%%%%%%
%\setcounter{section}{0}

%%% Now the paper itself

\section{Introduction.}

The question of homological characterization of projective spaces goes
back to Severi, and the pioneering work of Hirzebruch--Kodaira \cite{HK}.
Beautiful results have been obtained by 
Kobayashi--Ochiai \cite{KO}, Yau \cite{Yau}, Fujita \cite{Fuj}, Libgober--Wood \cite{LW}.

Among smooth projective varieties of given dimension projective spaces have 
the smallest %minimal possible 
cohomology groups.  
We call a smooth projective variety a {\it $\Q$-homology projective space} 
if it has the same Hodge numbers as a projective space.
Any odd-dimensional quadric is an example of $\Q$-homology projective space.
We call an $n$-dimensional $\Q$-homology projective space $X$ of general
type a {\it fake projective space} if in addition it has the same
``Hilbert polynomial'' %(with respect to canonical line bundle) 
as $\P^n$: 
$\chi(X,\omega_X^{\otimes l}) = \chi(\P^n,\omega_{\P^n}^{\otimes l})$
for all $l\in \Z$. 
Any {\it fake projective plane} is simply a $\Q$-homology plane of general type,
since Hodge numbers of a surface determine its Hilbert polynomial. 
On the level of realizations over $\C$, e.g. from the point of
view of the Hodge structure, fake projective spaces are identical to
projective spaces,
however the study of
% Note that all cohomological complications 
%with fake projective spaces or more generally
%$\Q$-homology projective spaces 
%lie in 
their $K$-theory, motive or
derived category
meets cohomological subtleties.

The first example of a fake projective plane
was constructed by Mumford \cite{FPP-Mu} using $p$-adic uniformization
developed by Drinfeld \cite{Dri} and Mustafin \cite{MUS}.
From the point of view of complex geometry fake projective planes have
been studied by Aubin \cite{Au} and Yau \cite{Yau}, who proved
that any such surface $S$ is uniformized by a complex ball,
hence by Mostow's rigidity theorem $S$ 
is determined by its fundamental group $\pi_1(S)$
uniquely up to complex conjugation;
Kharlamov--Kulikov \cite{KK} shown that
the conjugate surfaces are distinct (not biholomorphic).
Further Klingler \cite{Kl}
and Yeung \cite{Ye07}
proved that $\pi_1(S)$ 
%of a fake projective plane 
is a torsion-free cocompact \emph{arithmetic} subgroup of $PU(2,1)$.
Finally such groups have been classified by
Cartwright--Steger \cite{CS} and Prasad--Yeung \cite{PY07}:
there are $50$ explicit subgroups and
so all fake projective planes fit into
$100$ isomorphism classes.
% of complex conjugate pairs \cite{KK}.

Fake projective fourspaces were introduced and studied by Prasad and
Yeung in \cite{PY09}.

\medskip

In this paper we take a different perspective that started 
with a seminal discovery of \emph{full exceptional collections} 
by Beilinson \cite{Be}, %(on fake projective spaces)
Kapranov \cite{Ka}, % (on quadrics),
Bondal and Orlov \cite{BO} 
with Bondal's students Kuznetsov, Razin, Samokhin (see \cite{KuG2}):
they found out that all known to them examples of Fano $\Q$-homology projective spaces
admit a full exceptional collection of vector bundles.
They put a conjecture that gives a homological characterization of projective spaces based on derived categories,
and in this paper we prove it in Theorem \ref{mainthm}(3).

We call an $n$-dimensional smooth complex projective variety a {\it\bfseries minifold}
if it has a full exceptional collection of minimal possible length
$n+1$ in its bounded derived category of coherent sheaves.  A minifold
is necessarily a $\Q$-homology projective space.  Projective spaces
and odd-dimensional quadrics are examples of minifolds
\cite{Be,Ka}.

It follows from work of Bondal, Bondal--Polishchuk and Positselski
\cite{Bon90,BP,Pos}, that if a minifold $X$ is not Fano then all full
exceptional collections on it are not strict and consist not of pure
sheaves. In fact it is expected that all minifolds are Fano.

The novelty of this paper is the following  main theorem, which gives a classification of minifolds in
dimension less than or equal to $4$ (with one-dimensional case being trivial).

\begin{theorem}\label{mainthm}
1) The only two-dimensional minifold is $\P^2$.

2) The minifolds of dimension $3$ are:
%\begin{itemize}
%\item 
the projective space $\P^3$
%\item 
the quadric $Q^3$,
%\item 
the del Pezzo quintic threefold $V_5$,
%\item 
and
a six-dimensional family of Fano threefolds $V_{22}$.
%\end{itemize}

3) The only four-dimensional Fano minifold is $\P^4$.
\end{theorem}

In Section 2 we recall the necessary definitions and facts.
In particular in Proposition \ref{fec} we recall that
varieties admitting full exceptional collections have
Tate motives with rational coefficients \cite{MT} and outline
a straightforward proof of that fact. 
Section 2 finishes with the proof of Theorem \ref{mainthm}.  

We also show that except for $\P^4$ the only possible minifolds of
dimension $4$ are non-arithmetic fake projective fourfolds,
which presumably do not exist \cite{Ye-lectures} (paragraph 4 and section 8.4).

In fact, study of minifolds is closely related to study
of the fake projective spaces.
The reason is that fake projective spaces sometimes admit exceptional collections of the appropriate length
but these collections fail to be full. In this case the orthogonal 
to such a collection is a so-called phantom.

More precisely, we call an admissible non-zero subcategory $\AA$ of a derived category of coherent
sheaves an {\bf $H$-phantom} if $HH_*(\AA) = 0$ and
a {\bf $K$-phantom} if $K_0(\AA) = 0$.

\medskip

In Section $3$ we formulate a conjecture that under some mild conditions
fake projective $n$-spaces admit non-full exceptional collections of
length $n+1$ and thus have $H$-phantoms in their derived categories
(Conjecture \ref{conj-collections} and its Corollary
\ref{conj-phantoms}).  
We prove this conjecture for fake projective planes admitting an
action of the non-abelian group $G_{21}$ of order $21$.
% (Theorem \ref{theorem-fpp-G21}).
\begin{theorem}\label{theorem-fpp-G21}
Let $S$ be one of the six fake projective planes with automorphism
group of order $21$. Then $K_S = \OO(3)$ for a unique line bundle $\OO(1)$ on $S$.
Furthermore $\OO$, $\OO(-1)$, $\OO(-2)$ is an exceptional collection on $S$.
\end{theorem}
%This theorem provides new ways of constructing quasi-phantoms
%and together with the main theorem it puts the  foundations of and opens 
%a new area of research.
Most of Section 3 deals with the proof of this Theorem,
which relies on the holomorphic Lefschetz fixed point formula 
applied to the three fixed points of an element of order $7$ 
as in \cite{Ke09}.

\medskip

It follows from Theorem \ref{theorem-fpp-G21} that $\Db(S)$ has an $H$-phantom
subcategory $\AA_S$. We show that this $H$-phantom descends to
an $H$-phantom $\AA_S^G$ in the equivariant derived categories $\Db_G(S)$
for any $G \subset G_{21}$. 
In particular, when $G = \Z/3$ this gives yet new examples of surfaces 
having an $H$-phantom in their derived categories 
(fake cubic surfaces: see Remark \ref{fake-cubics}).

Finally, in four cases $G \supset \Z/7$ \emph{and} surface $S/G$ is simply-connected,
then we show that 
the $H$-phantom $\AA_S^G$ is also a $K$-phantom (Proposition \ref{K-phantoms}).

\medskip

In the Appendix we give a table of arithmetic subgroups $\Pi \subset PSU(2,1)$
giving rise to fake projective planes and the corresponding automorphism
and first homology groups. These results are taken from the computations
of Steger and Cartwright \cite{CScode}.

\medskip

%We would like to mention some open problems now.
%As already mentioned above it is expected that all minifolds are Fano, and moreover rational.
%It is not known whether minifolds can admit phantoms in their derived categories.
%One can expect that it is possible to prove more cases of Conjecture \ref{main-conjecture}
%using group actions or coverings of fake projective planes.

%\medskip

It took a long way for the paper
to take its present form. We would like to thank our friends
and colleagues with whom we had fruitful discussions on the topic.

We thank
Denis Auroux,		% comments to Ludmil
Alexey Bondal,		% for everything
Paul Bressler,		% K-theories and Chow groups
Alessio Corti,		%
Igor Dolgachev,		% 1. Lift from Mumford to V and consider G-representations
			% 2. Consider Mumford mod 2
Alexander Efimov,	% discussion on Pic and K_0 (Huatulco)
Sergey Gorchinskiy,     % discussion on no torsion in Picard
Jeremiah Heller,	% quoted result on Chow groups and $K_0$ of surfaces
Daniel Huybrechts,	%
Umut Isik,		% dos2unix etc
Yujiro Kawamata,	% Bangkok
Maxim Kontsevich,	% 
Viktor Kulikov,		% fundamental groups, also explanation of his results with Kharlamov
Alexander Kuznetsov, 	% reference to appendix in his paper, 
			% very careful proofreading, etc
Serge Lvovski,		% latex
Dmitri Orlov,		% discussion in Vienna (July 2011) of his unpublished work, 
			% reference to Prokhorov (Furushima),
			% discussion in Steklov (March 2013)
Dmitri Panov,		% idea to consider Keum's surface with automorphisms,
			% K\"ahler, Kodaira, etc
			% fundamental groups of quotients and resolutions
Tony Pantev,		% comments to Ludmil
Yuri Prokhorov,		% reference to Furushima,
Kyoji Saito,		% Oka principle
Konstantin Shramov,	% reference to Wilson
Duco van Straten,	% his computations with Spandau of Mumford modulo 2
Misha Verbitsky,	% fundamental groups, Bangkok
Vadim Vologodsky, 	% paid attention to denominators
and
Alexander Voronov	% torsion linking form
for their useful suggestions, references and careful proofreading.
We thank Donald Cartwright, Philippe Eyssidieux and Bruno Klingler, Gopal Prasad, Sai Kee Yeung, 
for answering our questions about fake projective planes and fourspaces.

\medskip

While this paper was in preparation, Najmuddin Fakhruddin gave a proof of Conjecture \ref{conj-collections}
for those fake projective planes that admit $2$-adic uniformization \cite{Fa},
in particular for Mumford's fake projective plane. 
These cases are disjoint from the ones that satisfy the assumptions of our
Theorem \ref{theorem-fpp-G21}.
An idea of a proof in the case of Mumford's fake projective plane similar to that of \cite{Fa}
was also hinted to us by Dolgachev.
Independently it was sketched to us by van Straten in June 2013:
van Straten and Spandau has an unpublished work circa 2001, where based on description of Ishida \cite{Ishida} 
they construct the reduction modulo two of the $2$-canonical image of Mumford's fake projective plane 
as an image of $\P^2(\F_2)$ by an explicit $10$-dimensional linear system of octics (plane curves of degree $8$),
one then checks that these octics are not divisible by $3$ as Weil divisors.
Our proof of Theorem \ref{theorem-fpp-G21} is indebted to discussions 
with Panov in June 2012 (he suggested us to exploit extra symmetries of Keum's surface)
and with Dolgachev in June 2013 (then we looked for higher-dimensional irreducible representations of non-abelian groups).
The rest of the ideas of the paper is from 2011.

\section{Minifolds}

An {\it exceptional collection}
 of length $r$
on a smooth projective variety $X/\C$ 
% in a triangulated category $\T$
is a sequence of objects $E_1, \dots E_r$ in % $\T$
the bounded derived category of coherent sheaves $\Db(X)$ 
such that $\Hom(E_j, E_i[k]) = 0$ for all $j > i$ and $k \in \Z$,
and moreover each object $E_i$ is exceptional, that is
spaces $\Hom(E_i,E_i[k])$ vanish for all $k$ except for
one-dimensional spaces $\Hom(E_i,E_i)$.  An exceptional collection is
called {\it full} if the smallest triangulated subcategory which
contains it, coincides with % $\T$ 
$\Db(X)$. 

\begin{proposition} \label{fec} Assume that $X$ admits a full
  exceptional collection of length $r$. 
Then:

(1) The Chow motive of $X$ with rational coefficients is a direct sum
of $r$ Tate motives $\L^j$.  In particular, all cohomology classes on
$X$ are algebraic.  (For the definition and properties of Chow motives
see \cite{Ma68}.)

(2) $H^{p,q}(X) = 0$ for $p \ne q$ and $\chi(X) = \sum h^{p,p}(X) = r$.

(3) $\Pic(X)$ is a free abelian group of finite rank.  Moreover the
first Chern class map gives an isomorphism $\Pic(X) \cong H^2(X,\Z)$.

(4) $H_1(X, \Z) = 0$.

(5) The Grothendieck group $K_0(X) = K_0(\Db(X))$ is free of rank $r$ and
the bilinear Euler pairing
\[
\chi(E,F) = \sum_i (-1)^i \dim \Hom(E,F[i])
\]
is non-degenerate and unimodular.
Classes $[E_i]$ of exceptional objects form a semi-orthonormal basis in $K_0(X)$.
(By a semiorthonormal basis we mean a basis $(e_i)_{i=0}^n$ such that
$\chi(e_j, e_i) = 0, \, j > i$ and $\chi(e_i,e_i) = 1$.)
\end{proposition}
\begin{proof}
Most of the claims are well-known.  (1) is proved in \cite{MT} using
the language of non-commutative motives and in \cite{GO} using
$K$-motives. We give a direct proof of (1) using the ideas developed in \cite{OrMotives} (see also \cite{BB})
for the sake of completeness.

First observe that the structure sheaf of the diagonal $\OO_\Delta$ in
the derived category $D^b(X \times X)$ lies in the full triangulated
subcategory generated by the objects $p_1^* F_1 \otimes p_2^* F_2$.
This can be deduced from the standard fact that if $E_1, \dots, E_r$
is a full exceptional collection on $X$, then $p_1^* E_i \otimes p_2^* E_j$
forms a full exceptional collection on $X \times X$ 
\cite{BVdB}(see Lemma 3.4.1 
and note that for the category generated by an exceptional collection 
the notions of generator and strong generator coincide,
so taking direct summands is not necessary), \cite{Boh}, \cite{Sam}.

It follows that the class of the diagonal $[\OO_\Delta] \in K_0(X \times X)$
has a decomposition
\begin{equation}\label{diagK}
[\OO_\Delta] = \sum_j p_1^* [\FF_j] \cdot p_2^* [\GG_j],
\end{equation}
for some sheaves $\FF_j, \GG_j$ on $X$.

Applying the Chern character to (\ref{diagK}) and using the Grothendieck-Riemann-Roch
formula
\[
ch(\OO_\Delta) = [\Delta] \cdot p_2^* td(X)
\]
we obtain an analogous decomposition for the class of the diagonal $[\Delta] \in CH^*(X \times X)_\Q$:
\begin{equation}\label{diagCH}
[\Delta] = \sum_j p_1^* \alpha_j \cdot p_2^* \beta^j,
\end{equation}
for some classes $\alpha_j, \beta_j \in CH^*(X)_\Q$.
We may assume that $\alpha_j$ are homogenous, say $\alpha_j \in CH^{a_j}(X)_\Q$
and hence $\beta_j \in CH^{dim(X)-a_j}(X)_\Q$.

We claim that the set $\{\alpha_j\}$ spans $CH^*(X)_\Q$.
Indeed for any $\alpha \in CH^*(X)_\Q$ we have
\begin{equation}\label{basisCH}\bal
\alpha &= p_{1*}( [\Delta] \cdot p_2^* \alpha ) = \\
&= p_{1*}( (\sum_j p_1^* \alpha_j \cdot p_2^* \beta^j) \cdot p_2^* \alpha  ) = \\
&= p_{1*}( \sum_j p_1^* \alpha_j \cdot p_2^* (\beta^j \cdot \alpha ) ) = \\
&= \sum_j \alpha_j \cdot p_{1*}( p_2^* (\beta^j \cdot \alpha ) ) = \\
&= \sum_j \langle \beta^j, \alpha \rangle \alpha_j.  \\
\eal\end{equation}

Here we use the notation $\langle \alpha, \beta \rangle$ for the bilinear form $deg(\alpha \cdot \beta)$.

We may assume that $\{\alpha_j\}$ are linearly independent,
that is form a homogeneous basis of $CH^*(X)_\Q$.
From the formula (\ref{basisCH}) we see that $\{\beta_j\}$ is
a dual basis.

We now define an isomorphism $M(X) \cong \oplus_{j} \L^a_j$.
By definition of the morphisms in the category of Chow motives we
have
\[\bal
Hom( \L^a, M(X) ) &= CH^a(M) \\
Hom( M(X), \L^a ) &= CH^{dim(X)-a}(M) \\
\eal\]

Therefore the set $\{\alpha_j\}$ determines a morphism of motives
\[
\Phi: \oplus_{j} \L^{a_j} \to M(X)
\]
and the $\{\beta_j\}$ determines a morphism in the opposite direction
\[
\Psi: M(X) \to \oplus_{j} \L^{a_j}.
\]

The composition $\Psi \circ \Phi$ is equal to identity due to the fact
that $\{\alpha_j\}$ and $\{\beta_j\}$ are dual bases.
The composition $\Phi \circ \Psi$ is equal to identity because
of the decomposition (\ref{diagCH}).

By taking Hodge realization (1) implies (2).
Alternatively, we can deduce (2) from Hochschild--Kostant--Rosenberg  theorem
\[
HH_i(\D^b(X)) \cong \oplus_{p-q=i} H^{p,q}(X)
\]
and additivity of Hochschild homology for semiorthogonal
decompositions \cite{Kel,KuHH}:
if $\CC=\langle\AA,\BB\rangle$ then $HH_i(\CC) = HH_i(\AA) \oplus HH_i(\BB)$. 

The fact that $\Pic(X)$ is free follows from (5) and Lemma
\ref{lemma-pic} below.  The isomorphism $\Pic(X) \cong H^2(X, \Z)$
comes from the exponential long exact sequence and (2).  $\Pic(X)$ is
of finite rank since it is isomorphic to $H^2(X, \Z)$.

To prove (4) note that
the Universal Coefficient Theorem implies that we have a non-canonical isomorphism
\[
H^2(X,\Z) \cong \Z^{\rk} \oplus H_1(X,\Z)^{tors}
\]
which by (3) implies that $H_1(X,\Z)$ must be torsion-free as well.
On the other hand $h^{1,0}(X) = 0$ and hence
$H_1(X,\Z) = 0$.

(5) follows easily from definitions.

\end{proof}

\begin{lemma} \label{lemma-pic}
Let $X$ be a smooth algebraic variety such that $K_0(X)$ has no $p$-torsion.
Then $\Pic(X)$ has no $p$-torsion.
\end{lemma}
\begin{proof}
We prove that if $\Pic(X)$ has $p$-torsion, then the same is true for $K_0(X)$.

Let $L$ be a line bundle on $X$ such that $L^{\otimes p} \cong \O_X$.
Let $N = [L] - 1 \in K_0(X)$; then $N$ is nilpotent.  Indeed $N$ being
of rank zero, sits in the first term $F^1 K_0(X)$ of the topological
filtration on $K_0(X)$ (\cite{Fulton}, Example $15.1.5$).  The
topological filtration is multiplicative; therefore $N^{dim(X)+1} \in
F^{dim(X)+1}K_0(X) = 0$.

Let $k$ be the smallest positive integer such that $N^k=0$.  If $k=1$,
that is $N=0$ and $[L] = 1 \in K_0(X)$, then $L \cong \O_X$ since to
$F^1 X/F^2 X \cong \Pic(X)$ by \cite{Fulton}, Example $15.3.6$.

We assume now that $k \ge 2$.
We have
\[
~[L] \= 1+N.
\]
Taking $p$-th tensor power of both sides we obtain
\[\bal
1    & \= 1 + pN + N^2 \alpha, \; \alpha \in K_0(X) \\
0    & \= pN + N^2 \alpha
\eal\]
and after multiplying by $N^{k-2}$:
\[
pN^{k-1} = 0,
\]
so that $N^{k-1}$ is nontrivial $p$-torsion class in $K_0(X)$.
\end{proof}

\medskip

\begin{remark} \label{kahler}
In fact if we assume $X$ to be a compact K\"ahler manifold with a full exceptional collection in the analytic
derived category $\D^b_{an}(X)$ of complexes of $\O_X$-modules with bounded coherent cohomology,
one can show that $H^{2,0}(X,\C) = 0$, so that
the K\"ahler cone is open in $H^2(X,\R) = H^{1,1}(X,\C) \cap \overline{H^{1,1}(X,\C)}$
hence it has non-trivial intersection with $H^2(X,\Z)$.

Then the Kodaira embedding theorem implies that $X$ is projective.
\end{remark}

\medskip

\begin{definition}
We call a smooth projective complex variety of dimension $n$ admitting a full exceptional collection of length $n+1$
a {\it minifold}.  
\end{definition}

It follows from Proposition \ref{fec} (2), that $n+1$ is the minimal
number of objects in such a collection, and the term "minifold" originates from here.
Minifolds have the same Hodge numbers as projective spaces.
By results of Beilinson \cite{Be} and Kapranov \cite{Ka}, projective spaces and odd-dimensional
quadrics are minifolds.

\begin{lemma} \label{ample}
Let $X$ be a minifold.
Then
\begin{itemize}
\item either $X$ is a Fano variety i.e. the anticanonical line bundle
  $\omega_X^\vee = \det T_X$ is ample 
\item or canonical line bundle $\omega_X = \det T_X^*$ is ample
\end{itemize}
In particular, the variety $X$ is uniquely determined by $\Db(X)$.
\end{lemma}
\begin{proof}
We first note that $\omega_X$ is not trivial, since $h^0(\omega_X) =
h^{n,0}(X) = 0$ by Proposition \ref{fec}(2). By Proposition
\ref{fec}(3), $\Pic(X)$ is torsion free; hence the class of $\omega_X$
in $\Pic(X)_\Q \cong H^{2}(X,\Q) = \Q$ is non-zero. Therefore either
$\omega_X$ or $\omega_X^\vee$ is ample.  Now the Bondal-Orlov
\cite{BO} reconstruction theorem implies the last statement.
\end{proof}

\medskip

\begin{remark} If we weaken the assumption from "projective" to
"proper" in the definition of a minifold, we still get the same
  class of varieties. Indeed, if $X$ is a proper smooth variety of dimension $n$
  with a full exceptional collection of length $n+1$ we can still
  deduce that $\omega_X$ or its dual is ample, in particular that $X$
  is projective as follows.

From \cite{LW}{(Theorem 3)} it follows that for a compact complex
$n$-dimensional manifold, the Chern number $c_1 c_{n-1}$ is determined
by Hirzebruch $\chi$-genera $\chi_y$ and hence by the Hodge numbers.

Thus we have $c_1 c_{n-1} [X] = c_1 c_{n-1} [\P^n] = \frac{n
  (n+1)^2}{2} \ne 0$.  Since the Kleiman-Mori cone of effective
one-cycles modulo numerical equivalence $N_1(X) \subset H^2(X,\R)$ is
one dimensional (that is because $H^2(X,\R)$ itself is one dimensional
by Proposition \ref{fec}(2) which still holds under the assumption
that $X$ is proper), Kleiman's criterion for ampleness implies that
either $\omega_X$ or its dual is ample.
\end{remark}

\medskip

The rest of this section is devoted to proof of Theorem \ref{mainthm}.
In view of Lemma \ref{ample}, the proof consists of
classifying Fano minifolds and showing that there is no minifolds
among varieties of general type.  

\medskip

We start in dimension $2$.  
The only del Pezzo surface with Picard number one is a projective plane.
%A two-dimensional Fano variety, that is a
%del Pezzo surface, is a blow up of $\P^2$ in $k \le 8$ points
% or a quadric.  From these surfaces $\P^2$ is the only minifold.

On the other hand it is known that fake projective planes have
non-vanishing torsion first homology group \cite{PY07}, Theorem 10.1.
Hence by Proposition \ref{fec}(4) there is no minifold of general type
of dimension $2$.

\medskip

Let us consider Fano threefolds.  By Proposition \ref{fec}(2) 
conditions $b_2(X)=1$ and $b_3(X)=0$ are necessary for a minifold.
Such Fano threefolds were classified by Iskovskikh \cite{Isk} into
four deformation types: the projective space $\P^3$, the quadric
$Q^3$, the del Pezzo quintic threefold $V_5$, and a family of Fano threefolds
$V_{22}$.

All these varieties are known to admit an exceptional collection of
length $4$ by results of Beilinson, Kapranov, Orlov and Kuznetsov respectively
\cite{Be}, \cite{Ka}, \cite{OrV5}, \cite{KuV22}.

\medskip

It is easy to see that $3$-dimensional $\Q$-homology varieties of general type do not exist.
Indeed $K_X$ ample implies that $c_1(X)^3$ is negative, but by Todd's theorem $c_1(X) c_2(X) = 24$.
This contradicts to Yau's inequality $c_1(X)^3 \ge \frac83 c_2(X) c_1(X)$ \cite{Yau}.

\medskip

According to Wilson \cite{Wi} and Yeung \cite{Ye10}
there are three alternatives for a
$\Q$-homology projective fourspace $X$:
 either $X$ is $\P^4$,
 or $X$ is a fake projective fourspace,
or $X$ has Hilbert polynomial $\chi(\omega_X^{-l}) = %P_X(l) = 
1 + \frac{25}{8} l (l+1) (3 l^2 + 3 l + 2)$
and Chern numbers
 $[c_1^4, c_2 c_1^2, c_2^2, c_1 c_3, c_4] = [225, 150, 100, 50, 5]$.
In what follows the varieties of the latter type are named \emph{Wilson's fourfolds}.

There are some known examples of fake projective fourfolds,
but it is not known whether any Wilson's fourfold actually exist.

In what follows we show that (possibly non-existent) Wilson's fourfolds
do not satisfy conditions of Proposition \ref{fec}(5),
and hence do not admit a full exceptional collection.
%even on the level of the Grothendieck group.
In order to do that we relate
the Grothendieck group of a minifold
to its Hilbert polynomial.

\medskip

We need a simple Lemma from linear algebra.

\begin{lemma} \label{lemma-det}
Let $P(x) = \sum_{j=0}^n p_j x^j \in K[x]$ be a polynomial of degree $\le n$ with coefficients in a field $K$ of characteristic zero
and let $A_P$
be the $(n+1)\times(n+1)$-matrix
with coefficients $a_{i,j} = P (j-i)$.
Then we have
\[
det(A_P) = (n! \,p_n)^{n+1}.
\]
In particular the matrix $A_P$ is non-degenerate
if and only if $\deg P = n$.
\end{lemma}

\begin{proof}
It suffices to prove the statement for algberaic closure $\bar{K}$ of $K$,
we thus assume $K$ to be algebraically closed.

We first prove that 
\beq\label{detAP}
det(A_P) = 0 \Longleftrightarrow p_n = 0.
\eeq
Indeed if $deg(P(x)) < n$, then $n+1$ polynomials $P(x), P(x+1), \dots, P(x+n)$ are
linearly dependent which makes the columns of $A_P$ linearly dependent, thus $det(A_P) = 0$.
On the other hand, it is easy to see that if $deg(P(x)) = n$, then
\[
P(x), P(x+1), \dots, P(x+n)
\]
form a basis of the space of polynomials of degree $\le n$, and $A_P$ is a matrix
of an invertible linear transformation $P \mapsto (P(0), P(-1), \dots, P(-n)) \in K^{n+1}$ in this basis,
hence $det(A_P) \ne 0$.

Let $F(p_0, p_1, \dots, p_n) = det(A_P)$. Since the entries of the matrix $A_P$ are linear forms in $p_0, p_1, \dots, p_n$,
it follows that $F$ is homogeneous in $p_i$'s of degree $n+1$.
Then (\ref{detAP}) says that the support of the degree $n+1$ hypersurface $F = 0$ in $\P^{n}$ is contained in the hyperplane $p_n = 0$.
Therefore 
\beq\label{FCn}
F(p_0, p_1, \dots, p_n)  = C_n \cdot p_n^{n+1}
\eeq
for some constant $C_n \in K$.
In particular $det(A_P)$ takes the same value $C_n$ for any monic polynomial $P(x)$ of degree $n$.

Let $P_0(x) = (x+1) \cdot (x+2) \cdot \dots \cdot (x+n)$. Then the matrix $A_{P_0}$ is
uppertriangular with all diagonal entries equal to $n!$:
\beq\label{Cn}
C_n = det(A_{P_0}) = (n!)^{n+1}
\eeq

The result now is the combination of (\ref{FCn}) and (\ref{Cn})
\end{proof}

\begin{comment}
\begin{proof}
Let $U,V$ be $(n+1)\times(n+1)$ matrices with entries
\[\bal
U_{i,j} &\= (-1)^{i+j}\binom{i-1}{j-1}, \\
V_{i,j} &\= (-1)^{i+j}\binom{n+1-j}{n+1-i}. \\
\eal\]
Note that $U,V$ are lower-unitriangular.
(Recall that  a binomial coefficient is zero if the bottom
number is strictly bigger than the top.)
 We compute the entries of $U A_P V$:

\[\bal
(U A_P V)_{i,j} &\=\sum_{k,l=1}^{n+1} (-1)^{i+k+l+j}  \binom{i-1}{k-1} P(l-k) \binom{n+1-j}{n+1-l}\\
&\=(-1)^{i+j}\sum_{k,l=0}^n (-1)^{k+l}P(l-k) \binom{i-1}{k} \binom{n+1-j}{n-l}\\
&\=(-1)^{i+j}\sum_{k,l=0}^n (-1)^{k+l+n}P(n-l-k)  \binom{i-1}{k} \binom{n+1-j}{l}.\\
\eal\]

Let $T$ be the shift operator on the space of polynomials, $(Tf)(x)=f(x-1)$. We have
\[\bal
(U A_P V)_{i,j} & \=
(-1)^{i+j}\left(\sum_{k,l=0}^n (-1)^{k+l+n} \binom{j-1}{k} \binom{n+1-i}{l} T^{l+k}  P\right)(n) \\
& \=(-1)^{i+j+n}\left((1-T)^{n-i+j}P\right)(n).
\eal\]
When $j>i$ we have $(1-T)^{n-i+j}P=0$, hence $U A_P V$ is lower-triangular.
When $i=j$ we have $(1-T)^{n-i+j}P=n! \, p_n$,
hence the diagonal entries of $U A_P V$ are equal to $(-1)^n n! \, p_n$.
Thus
$$
\det A_P = \det(U A_P V)= \left((-1)^n n! \, p_n\right)^{n+1} = (n! \, p_n)^{n+1}.
$$
\end{proof}
\end{comment}

\begin{proposition} \label{prop-reduction}
Let $X$ be a minifold.
Let $\O(1) = det(T_X)$ be the anticanonical bundle,
$\deg(X)$ be the anticanonical degree $c_1(X)^n$
and $P_X(k) = \chi(\O(k))$ be
the Hilbert polynomial.
Consider a sublattice $\Lambda \subset K_0(X)$ spanned by
\[
[\O], [\O(1)], \dots, [\O(n)].
\]

Then the Euler pairing restricted to $\Lambda$ is non-degenerate,
that is classes $[\O], [\O(1)], \dots, [\O(n)]$
are linearly independent in $K_0(X)$ and $\Lambda$ is a sublattice
in $K_0(X)$ of full rank.
Furthermore, $\Lambda$ admits a semi-orthonormal basis over
the ring $\Z[\frac{1}{deg (X)}]$
and hence modulo any prime $p$ that does not divide $deg(X)$.
\end{proposition}
\begin{proof}

Let $A_X$ denote the matrix of the pairing on $\Lambda$, that is a matrix with entries $a_{i,j} = \chi(\O(i), \O(j)) = P_X(j-i)$.

We apply Lemma \ref{lemma-det} to $P = P_X$, the Hilbert polynomial.
Its top coefficient is equal to $p_n = \frac{\deg(X)}{n!}$;
therefore $det(A_X) = deg(X) \ne 0$ is the anticanonical degree
and the pairing on $\Lambda$ is non-degenerate.

The inclusion $\Lambda \subset K_0(X)$ becomes an isomorphism
after inverting $det(A_X) = deg(X)$. Indeed
let $e_j, j=0, \dots n$ be a basis in $K_0(X)$
and write
\[
[\O(i)] = \sum G_{j,i} e_j, \; 0\leq i\leq n.
\]
The matrix $G^{-t} \, A_X \, G^{-1}$ is unimodular,
hence $deg(X) = det(G)^2$, and after
inverting $deg(X)$, $G$ becomes invertible.

Since $K_0(X)$ admits a semiorthonormal basis by assumption and Proposition \ref{fec}(5),
the same holds for $\Lambda \otimes \Z[\frac1{deg(X)}]$
\end{proof}

\medskip

Let $P_X$ be the Hilbert polynomial of Wilson's fourfolds
and $A_X$ be the $5\times 5$-matrix ${(A_X)}_{i,j} = P_X(j-i)$.
Consider their residues modulo two: $\overline{A}_X = A_X\mod 2$, 
${(\overline{A}_X)}_{i,j} = \overline{P}_X(j-i)$.

In the Proof of Proposition \ref{prop-reduction} we showed
that the determinant of matrix $A_X$ equals $deg(X)^{n+1} = 225^5=15^{10}$,
hence the assumption that $X$ is a minifold would imply
that $A_X$ admits a semiorthonormal basis modulo all primes $p \ne 3, 5$,
in particular this would imply that $\overline{A}_X$ has a semiorthonormal basis.

Entries of $A_X$ and $\overline{A}_X$ are determined by values $P(n)$ for $0\leq n \leq 4$ (that we tabulate)
and Serre duality $P(n) = P(-1-n)$:

\begin{center}
\begin{tabular}{|l|c|c|c|c|c|}
\hline
n   & 0 &  1 &   2 &    3 &    4 \\
\hline
$P(n)$ & 1 & 51 & 376 & 1426 & 3876 \\
\hline
$P(n)\mod 2$ & 1 & 1 & 0 & 0 & 0 \\
\hline
\end{tabular}
\end{center}
%We compute the matrix $A_X$ for a Wilson fourfold $X$:
%\[
%A_X = \begin{pmatrix}
% 1&51&376&1426&3876\cr
% 1&1&51&376&1426\cr
% 51&1&1&51&376\cr
% 376&51&1&1&51\cr
% 1426&376&51&1&1\cr
%\end{pmatrix}.
%\]
%Let

\[
\overline{A}_X %:= A_X \mod 2 
= \begin{pmatrix}
 1&1&0&0&0\cr
 1&1&1&0&0\cr
 1&1&1&1&0\cr
 0&1&1&1&1\cr
 0&0&1&1&1\cr
\end{pmatrix}.
\]

The following Lemma gives a contradiction, from which we see that a Wilson fourfold $X$ can not be a minifold.

\begin{lemma}\label{a2}
Let $(u,v)\mapsto u^t \overline{A}_X v$ be the bilinear form on a vector space $V = \F_2^5$ given
by the matrix $\overline{A}_X$.
There is no basis $e_1,e_2,e_3,e_4,e_5$ of $V$ such that $(e_i,e_j)=0$ for $i>j$ and $(e_i,e_i)=1$.
\end{lemma}
\begin{proof}

We begin by making a few remarks.

\begin{enumerate}
\item Let $S:=\overline{A}_X^{-1} \overline{A}_X^t$ be an automorphism of $V$.
In fact $S$ is induced by the Serre functor $\mathcal{S}_X = \otimes \omega_X [\dim X]$ on $\Db(X)$ \cite{Bon90, BK}.
$S$ satisfies $(u,v)=(v,S u)$ for all $u,v$,
so it preserves $\overline{A}_X$, i.e. $(u, v) = (Su, Sv)$,
equivalently $S^t \overline{A}_X S = \overline{A}_X$. We have
$$
S=\begin{pmatrix}
 1&1&0&0&0\cr
 0&0&1&0&0\cr
 0&0&0&1&0\cr
 1&0&0&0&1\cr
 1&0&0&0&0\cr
\end{pmatrix}
$$
and $S$ has order $8$ because the value of $P(n)\mod2$ depends only on $n\mod8$.
\item There are precisely $12$ vectors $x$ such that $(x,x)=1$.
Indeed $(x,x)=1$ if and only if the point $x$ does not lie on quadric $Q = \{x | (x,x)=0 \}$.
The quadric $Q$ has a unique singular point in $\P(V)$ so it has $19$ points
over $\F_2$ and its complement has $12$ points.
These twelve points form two orbits under the action of $S$. One orbit of length $8$ is generated by $a_1:=(1,0,0,0,0)^t$, another orbit of length $4$ is generated by $b_1:=(1,0,1,0,0)^t$.
\item If a basis $e_1,e_2,\ldots,e_5$ is semi-orthonormal, then for each $i$ ($1\leq i\leq 4$) the basis obtained by replacing $e_i, e_{i+1}$ with $e_{i+1}, e_i+e_{i+1} (e_i,e_{i+1})$ is also semi-orthonormal.
This transformation corresponds to mutations of exceptional collections \cite{Bon90,BK}.
\end{enumerate}

Denote $a_i=S^{i-1} a_1$, $b_i=S^{i-1} b_1$ and $c=(a_1,\ldots,a_8,b_1,\ldots,b_4)$. The following matrix has $(c_i,c_j)$ on position $i,j$:
$$
\left(\begin{array}{ccccccc|c|cccc}
 1&1&1&0&0&0&0&1&1&0&0&1\cr
 1&1&1&1&0&0&0&0&1&1&0&0\cr
 0&1&1&1&1&0&0&0&0&1&1&0\cr
 0&0&1&1&1&1&0&0&0&0&1&1\cr
 0&0&0&1&1&1&1&0&1&0&0&1\cr
 0&0&0&0&1&1&1&1&1&1&0&0\cr
\hline
 1&0&0&0&0&1&1&1&0&1&1&0\cr
 1&1&0&0&0&0&1&1&0&0&1&1\cr
\hline
 0&1&1&0&0&1&1&0&1&1&1&1\cr
 0&0&1&1&0&0&1&1&1&1&1&1\cr
 1&0&0&1&1&0&0&1&1&1&1&1\cr
 1&1&0&0&1&1&0&0&1&1&1&1\cr
\end{array}\right)
$$

\medskip

Assume there exists a semi-orthonormal basis. Then all of its vectors must be from the set $\{a_i\}\cup\{b_i\}$. Since there are only $4$ vectors in $\{b_i\}$, at least one of the basis vectors must be from $\{a_i\}$. Applying $S$ if necessary we may assume that this vector is $a_1$. Applying the transformation (3) we can obtain a semi-orthonormal basis with $a_1$ on the first position.

Any remaining basis vector $x$ must satisfy $(x,a_1)=0$.
Looking at the first column of the matrix of $(c_i,c_j)$ we see that the remaining basis vectors must be from the set $\{a_3,a_4,a_5,a_6,b_1,b_2\}$. Let $x$ be the second basis vector. Then any vector $y$ out of the remaining $3$ basis vectors must satisfy $(y,x)=0$. However, trying for $x$ each of the $\{a_3,a_4,a_5,a_6,b_1,b_2\}$ we see that there are only $2$ choices remaining for $y$. This is a contradiction.
\end{proof}

%\medskip
%\begin{remark}
%The Lemma \ref{a2} is in contrast with the following statement:
%a non-degenerate non-skew-symmetric bilinear form over a field (finite or infinite)
%with more than two elements admits a semi-orthogonal basis.
%\end{remark}

\medskip

We also can prove that there is no minifolds among arithmetic fake projective fourspaces.
This goes similarly to dimension $2$ case:
Prasad and Yeung proved that for an arithmetic fake projective fourspace
the first homology group $H_1(X,\Z)$ is non-zero
\cite{PY09}, Theorem $4$.
Therefore by Proposition \ref{fec}(3) these fourfolds are not minifolds.

\section{Phantoms in fake projective spaces}

Fake projective spaces seem to be very similar
and yet very different from ordinary projective spaces.
We propose the following conjecture.

\begin{conjecture} \label{conj-collections}
Assume that $X$ is an $n$-dimensional fake projective space with canonical class divisible
by $(n+1)$.
Then for some choice of $\O(1)$ such that $\omega_X = \O(n+1)$,
the sequence
\[
\OO, \OO(-1), \dots, \OO(-n)
\]
is an exceptional collection on $X$.
\end{conjecture}

We call an non-zero admissible subcategory $\AA \subset \D^b(X)$ an {\it\bfseries $H$-phantom} if $HH_\bullet(\AA) = 0$
and a {\it\bfseries $K$-phantom} if $K_0(\AA) = 0$.

\begin{corollary} \label{conj-phantoms}
Fake projective spaces as in Conjecture admit an $H$-phantom admissible
subcategories in their derived categories $\D^b(X)$. 
\end{corollary}
\begin{proof}
Assume that $\OO, \OO(-1), \dots, \OO(-n)$ is an exceptional collection,
and consider its right orthogonal $\AA$.
By results of Bondal and Kapranov \cite{Bon90,BK} the category $\AA$ is admissible, and thus we have
a semi-orthogonal decomposition:
\[
\D^b(X) = \left< \OO, \OO(-1), \dots, \OO(-n) , \AA \right>.
\]
Note that %all $\OO(i)$ are vector bundles , so 
this exceptional collection could not be full at least for two reasons:
\begin{itemize}
\item
if $\OO(i)$ would be a full collection then 
by \cite{BP}(Theorem 3.4) or \cite{Pos}(see proof of main theorem)
manifold $X$ would be Fano, which contradicts to general type assumption,
\item
use Corollary 4.6 and Proposition 4.7 of \cite{Ku12}:
by Kodaira vanishing for $i < j$ space $Ext^k(\OO(-i),\OO(-j))$ vanish unless $k=n$,
so relative height of any two objects in a helix $\OO(i)$ equals $n$,
thus pseudoheight of the collection coincides with its height and is equal to $n-1$,
hence Hochschild cohomology $HH^0(\AA) = HH^0(X) \neq 0$ for $n>1$.
\end{itemize}

Finally, Hochschild homology is additive for semi-orthogonal
decompositions (cf the alternative proof of \ref{fec}(2)), so $\dim HH_\bullet(\AA) = 0$ that is $\AA$ is an $H$-phantom.
\end{proof}

\begin{remark}\label{rem-collections}
1. A statement analogous to Conjecture \ref{conj-collections} holds for some fake del Pezzo surfaces
of degrees one \cite{BBS,BBKS}, six \cite{AO} and eight \cite{GS,Lee}.
Here we add degrees three (Remark \ref{fake-cubics}) and nine (Theorem \ref{theorem-fpp-G21}).

2. Fake projective planes with properties as in Conjecture \ref{conj-collections}
are constructed in \cite{PY07}, 10.4. Choose $\OO(1)$ such that $\OO(3) = \omega_X$.
Then by the Riemann-Roch theorem the Hilbert polynomial is given by
\[
\chi(\O(k)) = \frac{(k-1)(k-2)}{2}.
\]

Therefore the collection $E_\bullet = (\OO, \OO(-1), \OO(-2))$ is at least numerically exceptional,
that is
\[
\chi(E_j, E_i) = 0, \; j > i.
\]
In addition we have
\[
H^0(S, \O(1)) = H^0(S, \O(3)) = 0.
\]
Furthermore it follows from Serre duality
that a necessary and sufficient condition for
$E_\bullet$ to be exceptional is vanishing of the space of the global sections $H^0(S, \O(2))$.
It is not hard to see that for all fake projective planes $h^0(S, \OO(2)) \le 2$ (cf end of the Proof
of Theorem \ref{theorem-fpp-G21}).

\medskip

3. More generally our definition of an $n$-dimensional fake projective space includes that its Hilbert polynomial
is the same as that of a $\P^n$.
It follows that if we assume $\omega_X = \O(n+1)$, then we have
\[
\chi(\O(k)) = (-1)^n \frac{(k-1)(k-2) \dots (k-n)}{n!},
\]
so that $k = 1, \dots, n$ are the roots of $\chi$, and the collection
\[
\OO, \OO(-1), \dots, \OO(-n)
\]
is numerically exceptional.

4. G.Prasad and S.-K. Yeung informed us that the assumption $\omega_X = \O(5)$ is known to be true
for the four arithmetic fake projective fourspaces constructed in \cite{PY09}.

\end{remark}

We now prove Theorem \ref{theorem-fpp-G21},
which shows that conjecture \ref{conj-collections} holds
for fake projective planes admitting
an action of the non-abelian group $G_{21}$ of order $21$.

According to the Table given in the Appendix there are $6$ such
surfaces: there are three relevant groups in the table
and there are two complex conjugate surfaces for each group \cite{KK}.

%\begin{theorem}\label{theorem-fpp-G21}
%Let $S$ be one of the six fake projective planes with automorphism
%group of order $21$. Then $K_S = \OO(3)$ for a unique line bundle $\OO(1)$ on $S$.
%Furthermore $\OO$, $\OO(-1)$, $\OO(-2)$ is an exceptional collection on $S$.
%\end{theorem}

We first prove a general fact about fake projective planes.

\begin{lemma}\label{lemma-Kdiv}
Let $S$ be a fake projective plane with no $3$-torsion in $H_1(S, \Z)$.
Then there exists a unique (ample) line bundle $\OO(1)$ such that
$K_S \cong \OO(3)$. 
%If moreover $Aut(S)= G_{21}$, then
%$\OO(1)$ admits a $G_{21}$-linearization
%which is compatible with the natural $G_{21}$-linearization
%of $K_S$.
\end{lemma}
\begin{proof}
First note that the torsion in $\Pic(S) = H^2(S, \Z)$ is isomorphic
to $H_1(S, \Z)$ (cf Proof of Proposition \ref{fec}), 
hence $\Pic(S)$ has no $3$-torsion by assumption.

By Poincare duality $\Pic(S)/tors \cong H^2(S, \Z)/tors$ is a unimodular lattice,
therefore there exists an ample line bundle $L$ with $c_1(L)^2 = 1$.
Now $K_S - 3c_1(L) \in \Pic(S)$ is torsion which can be uniquely divided by $3$.
\end{proof}

\begin{proof}[Proof of Theorem \ref{theorem-fpp-G21}]
As follows from the classification of the fake projective planes by Prasad-Yeung and Cartwright--Steger, 
the order of the first homology group of the six fake projective planes with
automorphism group $G_{21}$ is coprime to $3$ (see the Table in the Appendix). 
Therefore by Lemma \ref{lemma-Kdiv} we have
\[
K_S = \OO(3).
\]
for a unique line bundle $\OO(1)$.

\medskip

Recall that $G_{21} = Aut(S) = N(\Pi)/\Pi$ where $N(\Pi)$ is a
normalizer of $\Pi$ in $PU(2,1)$ and by \cite{CSpriv}
the embedding 
\[
N(\Pi) \subset PU(2,1)
\] 
lifts to an embedding 
\[
N(\Pi) \subset SU(2,1)
\]
in all cases with $G_{21}$-action.
Therefore $\OO_B(-1)$ admits a $N(\Pi)$-linearization
and hence $\OO(1)$ admits a $G_{21}$-linearization,
compatible with the natural $G_{21}$-linearization of $K_S$.
We will consider vector spaces $H^*(S, \OO(k))$
as $G_{21}$-representations.

\medskip

According to Remark \ref{rem-collections}(2),
it suffices to show that $H^0(S, \OO(2)) = 0$.

\medskip

We now study the group $G_{21}$ and its representation theory.
By Sylow's theorems $G_{21}$ admits a unique subgroup of order $7$ and this subgroup is
normal. We let $\sigma$ denote a generator of this subgroup.
Let $\tau$ denote an element of $G_{21}$ of order $3$.
Conjugating by $\tau$ gives rise to an automorphism of $\Z/7 = \left< \sigma \right>$
and we can choose $\tau$ so that
\[
\tau^{-1} \sigma \tau = \sigma^2.
\]
Thus $G_{21}$ is a semi-direct product of $\Z/7$ and $\Z/3$ and has a presentation
\[
G_{21} = \left< \sigma, \tau \; | \; \sigma^7 = 1, \tau^3 = 1, \sigma \tau = \tau \sigma^2 \right>.
\]

Using this presentation it is easy to check that there are five conjugacy classes of elements in $G_{21}$:
\[
\{1\}
\]
\[
\{\sigma, \sigma^2, \sigma^4\}
\]
\[
\{\sigma^3, \sigma^5, \sigma^6\} 
\]
\[
\{\tau \sigma^k, \, k = 0,\dots,6\}
\]
\[
\{\tau^2 \sigma^k, \, k = 0,\dots,6\}
\]
and by basic representation theory there exist five irreducible representations of $G_{21}$. Let $d_1, \dots, d_5$
be the dimensions of these representations. Basic representation theory also tells us that each $d_i$ divides 21
and that
\[
d_1^2 + d_2^2 +d_3^2 +d_4^2 +d_5^2 = 21.
\]
Considering different possibilities one finds the only combination $(d_1, d_2, d_3, d_4, d_5) = (1,1,1,3,3)$
satisfying the conditions above.

%\newpage

It is not hard to check that the character table of $G_{21}$ is the following one:

\begin{longtable}{c|c|c|c|c|c|}

& $1$ & $[\sigma]$ & $[\sigma^3]$ & $[\tau]$ & $[\tau^2]$ \\

\hline
$\C$ &                    $1$ & $1$ & $1$ & $1$ & $1$ \\

\hline
$V_1$  &                $1$ & $1$ & $1$ & $\omega$ & $\overline{\omega}$ \\

\hline
$\overline{V_1}$ &    $1$ & $1$ & $1$ & $\overline{\omega}$ & $\omega$ \\

\hline
$V_3$ &                 $3$ & $b$ & $\overline{b}$ & $0$ & $0$ \\

\hline
 $\overline{V_3}$ &  $3$ & $\overline{b}$ & $b$ & $0$ & $0$ \\

\hline
\end{longtable}

\begin{comment}
     3  1  1  .  1  .
     7  1  .  1  .  1

       1a 3a 7a 3b 7b
    2P 1a 3b 7a 3a 7b
    3P 1a 1a 7b 1a 7a
    5P 1a 3b 7b 3a 7a
    7P 1a 3a 1a 3b 1a

X.1     1  1  1  1  1
X.2     1  A  1 /A  1
X.3     1 /A  1  A  1
X.4     3  .  B  . /B
X.5     3  . /B  .  B

A = E(3)^2
  = (-1-Sqrt(-3))/2 = -1-b3
B = E(7)+E(7)^2+E(7)^4
  = (-1+Sqrt(-7))/2 = b7
\end{comment}

Here we use the notation:
\[
\omega = e^{\frac{2 \pi i}{3}}
\]
\[
\xi = e^{\frac{2 \pi i}{7}}
\]
and
\[
b = \xi + \xi^2 + \xi^4 =  \frac{-1+\sqrt{-7}}{2}.
\]

Explicitly $V_1$ and $\overline{V_1}$ are one-dimensional
representations restricted from
 $G_{21}/\left< \sigma \right> = \Z/3$.
$V_3$ and $\overline{V_3}$ are three-dimensional representations induced from $\Z/7$:
$\rho: G_{21} \to GL(V_3)$ is given by matrices
\[
\rho(\sigma) = \left(\begin{array}{ccc}
\xi & & \\
& \xi^2 & \\
& & \xi^4 \\
\end{array}\right)  \;\;\;
\rho(\tau) = \left(\begin{array}{ccc}
0 & 0 & 1\\
1 & 0 & 0 \\
0 & 1 & 0\\
\end{array}\right)
\]
and $\overline{V_3}$ is its complex conjugate.

\begin{lemma}
$H^0(S, \OO(4))$ is a 3-dimensional irreducible representation
of $G_{21}$ (and thus is isomorphic to $V_3$ or $\overline{V_3}$).
\end{lemma}
\begin{proof}
We show that the trace of an element $\sigma \in G_{21}$ of order $7$
acting on $H^0(S, \OO(4))$ is equal to $b$ or $\overline{b}$.
This is sufficient since if $H^0(S, \OO(4))$ were reducible it would have to be
a sum of three one-dimensional representations and the character table of $G_{21}$
shows that in this case the trace of $\sigma$ on $H^0(S, \OO(4))$ would be equal
to $3$.

By \cite{Ke09}, Proposition 2.4(4) $\sigma$ has three fixed points $P_1$, $P_2$, $P_3$.
Let $\tau$ be an element of order $3$. 
$\tau$ does not stabilize any of the $P_i$'s, since a tangent space of a fixed point of 
$G_{21}$ would give a faithful $2$-dimensional representation of $G_{21}$ which does not
exist as is seen from its character table.

Thus $P_i$'s are cyclically permuted by $\tau$. We reorder $P_i$'s in such a way that
\beq\label{tau-Pi}
\tau(P_i) = P_{i+1 \; \text{mod} \; 3}.
\eeq

We apply the so-called Holomorphic Lefschetz Fixed Point Formula (Theorem 2 in \cite{AB})
to $\sigma$ and line bundles $\OO(k)$:
\beq\label{HLFP}
\sum_{p=0}^2 (-1)^p \, Tr\bigl(\sigma\big|_{H^p(S, \OO(k))}\bigr) = \sum_{i=1}^{3} \frac{Tr(\sigma|_{\OO(k)_{P_i}})}{(1-\alpha_1(P_i))(1-\alpha_2(P_i))}
\eeq
where $\alpha_1(P_i)$, $\alpha_2(P_i)$ are inverse eigenvalues of $\sigma$ on $T_{P_i}$:
\[
det(1-t\sigma_*|_{T_{P_i}}) = (1-t\alpha_1(P_i))(1-t\alpha_2(P_i)).
\]

$\alpha_j(P_i)$ are $7$-th roots of unity. We let $\alpha_j := \alpha_j(P_1)$, $j=1,2$.
Using (\ref{tau-Pi}) and commutation relations in $G_{21}$ we find that
\[
\alpha_j(P_{i+1}) = \alpha_j(P_i)^2
\]
so that
\[\bal
\alpha_j(P_1) &= \alpha_j \\
\alpha_j(P_2) &= \alpha_j^2 \\
\alpha_j(P_3) &= \alpha_j^4. \\
\eal\]

To find the values of $\alpha_j$ we apply (\ref{HLFP}) with $k=0$:
\beq\label{HLFP0}
1 = \frac1{(1-\alpha_1)(1-\alpha_2)} + \frac1{(1-\alpha_1^2)(1-\alpha_2^2)} + \frac1{(1-\alpha_1^4)(1-\alpha_2^4)}.
\eeq
All $\alpha_j(P_i)$ are $7$-th roots of unity and it turns out that up to renumbering the only possible values of $\alpha_j(P_i)$ which
satisfy (\ref{HLFP0}) are
\[\bal
(\alpha_1(P_1), \alpha_2(P_1)) &= (\xi, \xi^3)\\
(\alpha_1(P_2), \alpha_2(P_2)) &= (\xi^2, \xi^6)\\
(\alpha_1(P_3), \alpha_2(P_3)) &= (\xi^4, \xi^5)\\
\eal\]
or their complex conjugate in which case we would get $b$ instead of $\overline{b}$ for the trace below.

It follows that $Tr(\sigma|_{K_{S,P_i}}) = Tr(\sigma|_{\OO(3)_{P_i}})$ is equal to $\xi^4$, $\xi$, $\xi^2$ for
$i = 1, 2, 3$ respectively.
Dividing by $3$ modulo $7$ we see that $Tr(\sigma|_{\OO(k)_{P_i}})$ is equal to $\xi^{6k}$, $\xi^{5k}$, $\xi^{3k}$ for
$i = 1, 2, 3$ respectively.

We use (\ref{HLFP}) for $k=4$ (note that $H^p(S, \OO(4)) = 0$ for $p>0$ by Kodaira vanishing):
\[\bal
Tr\bigl(\sigma\big|_{H^0(S, \OO(4))}\bigr) &= \frac{\xi^3}{(1-\xi)(1-\xi^3)} + \frac{\xi^6}{(1-\xi^2)(1-\xi^6)} + \frac{\xi^5}{(1-\xi^4)(1-\xi^5)} = 
\overline{b}\eal\]
%\\
%&= \frac{1}{(1-\xi)(\xi^{-3}-1)} + \frac{1}{(1-\xi^2)(\xi^{-6}-1)} + \frac{1}{(1-\xi^4)(\xi^{-5}-1)} = \\
%&= -\frac{1}{(1-\xi)(1-\xi^4)} - \frac{1}{(1-\xi^2)(1-\xi)} - \frac{1}{(1-\xi^4)(1-\xi^2)} = \\
%&= -\frac{\xi}{(\xi^{-4}-\xi^{4})(\xi^{-2}-\xi^2)} - \frac{\xi^2}{(\xi^{-1}-\xi)(\xi^{-4}-\xi^4)} - \frac{\xi^4}{(\xi^{-2}-\xi^2)(\xi^{-1}-\xi)} = \\
%&= -\frac{3-b}{(\xi^{-1}-\xi)(\xi^{-2}-\xi^2)(\xi^{-4}-\xi^4)}. 
%\eal\]
%%}

%It is easy to compute that
%\[\bal
%(\xi^{-1} - \xi)(\xi^{-2} - \xi^2)(\xi^{-4} - \xi^4) = (\xi^6 - \xi)(\xi^5 - \xi^2)(\xi^3 - \xi^4) = 
%%&= (\xi^4 + \xi^3 - \xi^6 - \xi)(\xi^3 - \xi^4) = \\
%%&= 1 + \xi^6 + \xi^3 + \xi^5 - \xi^2 - \xi^4 - \xi - 1 = \\
%%&= \xi^6 + \xi^3 + \xi^5 - \xi^2 - \xi^4 - \xi  = \\
%%= -2b - 1 = 
%- \sqrt{-7}
%\eal\]
%so that the we can compute the trace:
%\[
%Tr\bigl(\sigma\big|_{H^0(S, \OO(4))}\bigr) = \frac{3-b}{\sqrt{-7}} = \frac{7-\sqrt{-7}}{2\sqrt{-7}} = \frac{-\sqrt{-7}-1}{2} = \overline{b}.
%\]
\end{proof}

\medskip

We are now ready to show that $H^0(S, \OO(2)) = 0$. Let $\delta = h^0(S, \OO(2))$. 
We know that $h^0(S, \OO(4)) = 3$, hence it follows from Lemma \ref{lemma-lin-systems} 
applied to $L = L' = \OO(2)$ that $\delta \le 2$.
Therefore as a representation of $G_{21}$ the space $H^0(S, \OO(2))$ is a sum of $1$-dimensional
representations and the same is true for $H^0(S, \OO(2))^{\otimes 2}$. Since $H^0(S, \OO(4))$ is three-dimensional
irreducible, this implies that the natural morphism
\[
H^0(S, \OO(2))^{\otimes 2} \to H^0(S, \OO(4))
\]
has to be zero by Schur's Lemma. Now again by Lemma \ref{lemma-lin-systems} $H^0(S, \OO(2)) = 0$.
This finishes the proof of Theorem \ref{theorem-fpp-G21}.
\end{proof}

\begin{lemma}[see \cite{Kol95}(Lemma 15.6.2)] \label{lemma-lin-systems}
%1. 
%Let $L, L' \in \Pic(X)$ be a pair of effective line bundles a variety $X$.
Let $X$ be a normal and proper variety,
$L$,$L'$ effective line bundles on $X$.
Let 
\[
\phi: H^0(X,L) \otimes H^0(X,L') \to H^0(X,L \otimes L')
\]
denote the natural map induced by multiplication. Then 
\[ 
\dim Im (\phi) \geq h^0(X,L) + h^0(X,L') - 1.
\] 
%
%2. Let $L \in \Pic(X)$ be an effective line bundle on $X$. For any $k\geq 2$ consider the self-multiplication map 
%\[
%\psi_k: S^k H^0(X, L) \to H^0(X, L^{\otimes k}).
%\]
%Then 
%\[ 
%\dim Im (\psi_k) \geq k (h^0(X,L) - 1) + 1 .
%\]
\end{lemma}
% the proof is commented since we put a reference to Lemma 15.6.2 in Kollar's book
%\begin{proof} 
%Note that any non-zero decomposable element $s_1 \otimes s_2$ does not lie in $Ker(\phi)$. 
%The set of decomposable elements is a cone over the image of the Segre embedding
%\[
%\P(H^0(X, L)) \times \P(H^0(X, L')) \subset \P(H^0(X, L) \otimes H^0(X, L')).
%\]
%Since $\P(Ker(\phi))$ does not intersect this subvariety it follows that
%\[
%\codim(\P(Ker(\phi)) \ge \dim \bigl( \P(H^0(X, L)) \times \P(H^0(X, L')) \bigr) + 1 = h^0(X, L) + h^0(X, L') - 1.
%\]
%However
%\[
%\dim(Im(\phi)) = \codim(Ker(\phi)) = \codim(\P(Ker(\phi)))
%\]
%and this finishes the proof of the Lemma.
%\end{proof}

\medskip
\medskip

We now consider equivariant derived categories $\Db_G(S)$ for
various subgroups $G \subset G_{21}$.
A good reference for equivariant derived categories and
their semi-orthogonal decompositions is \cite{El}.

It is easy to see that
\beq\label{equiv-collections}
\{\OO(-j) \otimes V\}_{j = 0,1,2; \; V \in IrrRep(G)}
\eeq
forms an exceptional collection in the equivariant derived category $\Db_G(S)$.
We denote by $\AA_S^G$ the right orthogonal to this collection.

It is easy to see that the category $\AA_S^G$ is non-zero.
This follows from the Kuznetsov's criterion (cf the second proof of Corollary \ref{conj-phantoms})
since the height of the exceptional collection equals $n-1$.
We also notice that for any nonzero object $A$ in $\AA_S$ the object
\[
\bigoplus_{g\in G} g^* A
\]
will have a natural $G$-linearization so will be a non-zero object in $\AA_S^G$.

\begin{proposition}\label{equiv-H-phantoms}
Let $S$ be a fake projective plane with automorphism group $G_{21}$.
For any $G \subset G_{21}$, $\AA_S^G$ is an $H$-phantom.
\end{proposition}
\begin{proof}

We denote by $Z_G$ the minimal resolution of $S/G$. The geometry of $Z_G$ has been
carefully studied by Keum \cite{Ke09}: if $|G|=7$ or $|G|=21$ then $Z_G$ is an elliptic surface of Kodaira dimension $\kappa(Z_G) = 1$ (Dolgachev surface),
if $|G|=3$ then $Z_G$ is a surface of general type $\kappa(Z_G)=2$.
In each case we compare the equivariant derived category $\Db_G(S)$ to $\Db(Z_G)$.

The stabilizers of the fixed points of $G$ action are cyclic and we use \cite{IU} or \cite{Kaw12} 
to obtain the semi-orthogonal decomposition
\[
\Db_G(S) \simeq \left< \Db(Z_G), E_1, \dots, E_{r_G} \right>
\]
where $r_G$ is the number of non-special characters of the stabilizers \cite{IU}.
%{\zhenya Careful here: how do we make a global statement?}
%{\sergey Similarly to Kapranov-Vasserot: there is a functor defined by a kernel, we need to check local things ...}

Note that $p_g(Z_G) = q(Z_G) = 0$, therefore
\[
\dim HH_*(\Db(Z_G)) = \dim H^*(Z_G, \C) = \chi(Z_G).
\]

%Next newpage is to keep table on one page. Remove if something else is changed.
\newpage
We list $\chi(Z_G)$ as well as other relevant invariants in the table:

\begin{longtable}{c|c|c|c|c|c|}

$G$ & $\#IrrRep(G)$ & $Sing(S/G)$ & $r_G$ & $\chi(Z_G)$ & $\kappa(Z_G)$  \\

\hline
$1$ &			$1$ & $\emptyset$ & $0$ & $3$ & $2$ \\

\hline
$\Z/3$ &              $3$ & $3 \times \frac13(1,2)$ & $0$ & $9$ & $2$ \\

\hline
$\Z/7$  &             $7$ & $3 \times \frac17(1,3)$ & $9$ & $12$ & $1$ \\

\hline
$G_{21}$ &         $5$ & $3 \times \frac13(1,2) + \frac17(1,3)$ & $3$ & $12$ & $1$ \\

\hline
\end{longtable}

As already mentioned above $r_G$ is the sum of non-special characters of
the stabilizers at fixed points: $\frac13(1,2)$ fixed points don't contribute
to $r_G$ whereas each $\frac17(1,3)$ fixed point has $3$ non-special characters.

\medskip

It follows from the table that in each case we have
\[
3 \cdot \#IrrRep(G) = \chi(Z_G) + r_G 
\]
This implies that the number of exceptional objects in (\ref{equiv-collections}) matches
$\dim HH_*(\Db_G)$, and therefore in each case $\AA_S^G$ is an $H$-phantom.
\end{proof}

\begin{remark}\label{fake-cubics}
When $G = \Z/3$, $r_G = 0$ means that
\[
\Db_G(S) \simeq \Db(Z_G)
\]
in agreement with the derived McKay correspondence \cite{KV,BKR}
which is applicable since $S/G$ has $A_2$ singularities.
$Z_G$ is a fake cubic surface ($p_g(Z_G) = q(Z_G) = 0$, $b_2(Z_G) = 7$)
and the image of the exceptional collection (\ref{equiv-collections}) of $9$ objects
in $\Db(Z_G)$ has an $H$-phantom orthogonal.
\end{remark}

\begin{remark}
One can give an alternative proof of Proposition \ref{equiv-H-phantoms}
using orbifold cohomology.
Baranovsky \cite{Bara} proved an analogue of Hochschild--Kostant--Rosenberg isomorphism for orbifolds.
His result implies that (total) Hochschild homology $HH_*(\Db_G(S))$
is isomorphic as a non-graded vector space to the (total) orbifold cohomology 
\[
H^*_{orb}(S/G, \C) = \bigl(\bigoplus_{g \in G} H^*(S^g, \C) \bigr)_G = \bigoplus_{[g] \in G/G} H^*(S^g, \C)^{Z(g)}.
\]
Here $S^g$ is the fixed locus of $g \in G$, 
$Z(g)$ is the centralizer,
$[g]$ is the conjugacy class of $g$,
and $(\cdot)_G$ denotes coinvariants.
In our case the following two assumptions are satisfied:
\begin{itemize}
\item group $G$ acts trivially on $H^*(S, \C)$ (thanks to minimality of $S$),
\item for each element $g\neq 1$ its fixed locus $S^g$ is a union of $\dim H^*(S,\C)$ points (this is usually derived from Hirzebruch proportionality principle, see e.g. \cite{KK,Ke09}).
\end{itemize}

For the so-called main sector $[g]=\{id\}$ we have
\[
H^*(S, \C)^G = H^*(S, \C) = \C^3.
\]
%since the action on the cohomology is trivial by minimality.
For each $g\neq id$ the fixed locus $S^g$ consists of three points,
so $H^*(S^g) = \C^3$ and the action of $Z(g) = \langle g\rangle$ on it is trivial,
thus each twisted sector is also $3$-dimensional. 

Taking the sum over all conjugacy classes $[g]$ we obtain
\[
\dim HH_*(\Db_G(S)) = \dim H^*_{orb}(S/G, \C) = 3 \times \#IrrRep(G),
\]
which shows that $HH_*(\AA_S^G) = 0$.
\end{remark}

\begin{proposition}\label{K-phantoms}
Let $S$ be a fake projective plane with automorphism group $G_{21}$.
In the notations of \cite{PY07,CS} and the appendix assume that the class of $S$ is either
$(\Q(\sqrt{-7}), p=2, \T_{1} = \{7\})$ or $\CC_{20}$.
Let $G = \Z/7 \subset G_{21}$ or $G = G_{21}$.
Then the orthogonal to the collection (\ref{equiv-collections}) in $\Db_G(S)$
is a $K$-phantom.
\end{proposition}
\begin{proof}
Let $\Pi_G \subset \overline{\Gamma}$ be the group generated by $\pi_1(S)$ and $G$.
By \cite{Bar}(0.4) the fundamental group $\pi_1(S/G)$ equals to $\Pi_G/E$
where $E \subset \Pi_G$ is the subgroup generated by elliptic elements of $\Pi_G$ i.e.
elements $\gamma \in \Pi_G$ such that fixed locus $B^\gamma \neq \emptyset$ is non-empty.
Cartwright and Steger in \cite{CSpriv} explicitly computed $E$ and so $\pi_1(S/G)$ for various
subgroups $\Pi \subset \Pi_G \subset \overline{\Gamma}$ and it turns out that in the cases under consideration
the quotients $S/G$ are simply connected.
By a standard argument 
(e.g. using Van Kampen's theorem as in \cite{Bar}(0.5) or \cite{Ver}(Section 4.1), or more generally see \cite{Kol93}(Theorem 7.8.1)) the resolutions $Z_G$ are also simply-connected,
in particular $H_1(Z_G, \Z) = 0$.

Then $\Pic(Z_G) = H^2(Z_G, \Z)$ is finitely generated free abelian group.
Keum shows in \cite{Ke09} that Kodaira dimension $\kappa(Z_G) = 1$ (see also Ishida \cite{Ishida}).
Thus Bloch conjecture for $Z_G$ is true by Bloch--Kas--Lieberman \cite{BKL},
that is $CH_0(Z_G) = \Z$. 
Now by Lemma 2.7 of \cite{GS} it follows that $K_0(Z_G)$ is a finitely generated 
free abelian group, and the same holds for $K_0^G(S) = K_0(\Db_G(S))$.

The computation of Euler numbers shows that
\[
\text{(number of objects in (\ref{equiv-collections}))} = \dim HH_*(\Db_G(S)) = \rk K_0(\Db_G(S)).
\]
Finally, the additivity of the Grothendieck group implies that $\AA_S^G$ is a $K$-phantom.
\end{proof}

\section*{Appendix: automorphisms and first homology
groups of fake projective planes}

Recall that all fake projective planes $S$ are quotients of a complex ball $B \subset \C\P^2$
by a cocompact torsion-free arithmetic subgroup $\Pi = \pi_1(S)$ \cite{PY07}, \cite{CS},
and each of the fifty possible groups $\Pi$ corresponds to 
a pair complex conjugate surfaces $S$ and $\overline{S}$ which are not isomorphic
to each other \cite{KK}.
The first homology group $H_1(S,\Z)$ of $S$ is isomorpic to the abelianisation of the 
$\Pi / [\Pi,\Pi]$ and the automorphism group equals $\Aut(S) = N(\Pi) / \Pi$,
where $N(\Pi)$ is the normaliser of $\Pi$ (in maximal arithmetic group $\overline{\Gamma}$
and hence in any group that contains it, in particular in $PU(2,1)$).

%$50$ groups $\Pi$ fall into $28$ classes (number of groups in a class varies from one to six),
%each class is enumerated by number-theoretic data
%$(k,l,p, \T_1, N)$:
%\begin{itemize}
%\item $k$ is a totally real number field ($\Q$ or $\Q(\sqrt{d})$ for $d\in\{2,5,6,7\}$),
%\item $l$ is a totally complex quadratic extension of $k$,
%\item $p$ is a prime number ($2$, $3$ or $5$)
%\item $\T_1$ is a set of prime numbers (possibly empty)
%\item $N$ is a positive integer ($1$, $3$, $9$ or $21$), it equals to the index %$[\overline{\Gamma} : \Pi]$
%\end{itemize}

%\newpage

We enhance the classification table of the fake projective planes given in \cite{CS} 
which is based on GAP and Magma computer code and
its output \cite{CScode} 
with the automorphism group $Aut(S)$ and the first homology group $H_1(S, \Z)$, 
which we also take from \cite{CScode}.

In the table $\overline{\Gamma}$ is described using the following data:
$l$ is a totally complex quadractic extension of a totally real field,
$p$ is a prime $2$, $3$ or $5$, 
$\T_1$ is a set of prime numbers (possibly empty).

$N$ is the index $[\overline{\Gamma} : \Pi]$ and $suf.$ is
the suffix ($a,b,c,d,e$ or $f$) of each group in \cite{CScode}.
$G_{21}$ is the non-abelian group of order $21$.
In the last column symbol $[n_1,\dots,n_k]$ denotes the abelian group 
$(\Z/n_1 \Z)\times\dots\times(\Z/n_k\Z)$.

%Note that since the normalizer $N_{\overline{\Gamma}}(\Pi)$ is an 
%intermediate subgroup 
%$\Pi \subset N \Pi \subset \overline{\Gamma}$ we see that $|\Aut M| | N$.
Consider the quotient-map $f: N(\Pi) \to N(\Pi)/\Pi = \Aut (S)$ and
for a subgroup $G \subset \Aut (S) = N(\Pi)/\Pi$
let $\Pi_G \subset N(\Pi) \subset \overline{\Gamma}$
be the preimage $\Pi_G = f^{-1} G$. Line bundle $\O_S(1)$
is $G$-linearisable $\iff$ group $\Pi_G$ lifts from $PU(2,1)$
to $SU(2,1)$. Computation of Cartwright and Steger \cite{CSpriv} shows that
it holds for all $S$ and $G$ unless group $\overline{\Gamma}$ lies in classes $\CC_{2}$ or $\CC_{18}$. 
Fundamental group of the quotient-surface $\pi_1(S/G)$ equals $\Pi_G/E$ where $E \subset \Pi_G$ is the subgroup generated by elliptic elements
(cf the proof of Proposition \ref{K-phantoms}).
All those groups for all $S$ and $G \subset \Aut (S)$ were also computed in \cite{CSpriv}:
surface $S/G$ is simply-connected in twelve cases, including the four cases of Proposition \ref{K-phantoms}.

%\bigskip
\newpage

{\tiny 
\begin{longtable}{|c|c|c|c|c||c|c|c|}

\hline
$l$ or $\CC$ & $p$ & $\T_1$ & $N$ & $\#\Pi$ &
$suf.$ & $\Aut(S)$ & $H_1(S,\Z)$ \\
%& $SU?$ \\

\hline
$\Q(\sqrt{-1})$ & $5$ & $\emptyset$ & $3$ & $2$ & 
$a$ & $\Z/3\Z$ & $[2,4,31]$   \\

& & $\emptyset/\{2I\}$ & & & 
$b/b$ & $\{1\}$ & $[2,3,4,4]$   \\

& & $\{2\}$ & $3$ & $1$ & 
$a$ & $\Z/3\Z$ & $[4,31]$  \\

%& & & & 
%b$ & $\{1\}$ & $[2,3,4,4]$ \\

\hline

$\Q(\sqrt{-2})$ & $3$ & $\emptyset$ & $3$ & $2$ & 
$a$ & $\Z/3\Z$ & $[2,2,13]$  \\

& & $\emptyset/\{2I\}$ & & & 
$b/b$ & $\{1\}$ & $[2,2,2,2,3]$ \\

& & $\{2\}$ & $3$ & $1$ & 
$a$ & $\Z/3\Z$ & $[2,2,13]$ \\

%& & & & 
%$b$ & $\{1\}$ & $[2,2,2,2,3]$ \\

\hline

$\Q(\sqrt{-7})$ & $2$ & $\emptyset$ & $21$ & $3$ & 
$a$ & $\Z/3\Z$ & $[2,7]$ \\

& & & & & 
$b$ & $G_{21}$ & $[2,2,2,2]$ \\

& & & & & 
$c$ & $\{1\}$ & $[2,2,3,7]$  \\

& & $\{3\}$ & $3$ & $2$ & 
$a$ & $\Z/3\Z$ & $[2,4,7]$  \\

& & & & & 
$b$ & $\{1\}$ & $[2,2,3,4]$  \\

& & $\{3,7\}$ & $3$ & $2$ & 
$a$ & $\Z/3\Z$ & $[4,7]$  \\

& & & & & 
$b$ & $\{1\}$ & $[2,3,4]$  \\

& & $\{7\}$ & $21$ & $4$ & 
$a$ & $G_{21}$ & $[2,2,2]$  \\

& & & & & 
$b$ & $\Z/3\Z$ & $[2,7]$  \\

& & & & & 
$c$ & $\Z/3\Z$ & $[2,2,7]$  \\

& & & & & 
$d$ & $\{1\}$ & $[2,2,2,3]$  \\

& & $\{5\}$ & $1$ & $1$ & 
$-$ & $\{1\}$ & $[2,2,9]$  \\

& & $\{5,7\}$ & $1$ & $1$ & 
$-$ & $\{1\}$ & $[2,9]$  \\

\hline

$\Q(\sqrt{-15})$ & $2$ & $\emptyset$ & $3$ & $2$ & 
$a$ & $\Z/3\Z$ & $[2,2,7]$  \\

& & & & & 
$b$ & $\{1\}$ & $[2,2,2,9]$ \\

& & $\{3\}$ & $3$ & $3$ & 
$a$ & $\Z/3\Z$ & $[2,3,7]$ \\

& & & & & 
$b$ & $\Z/3\Z$ & $[2,2,2,3]$ \\

& & & & & 
$c$ & $\Z/3\Z$ & $[2,3]$ \\

& & $\{3,5\}$ & $3$ & $3$ & 
$a$ & $\Z/3\Z$ & $[3,7]$ \\

& & & & & 
$b$ & $\Z/3\Z$ & $[2,2,3]$ \\

& & & & & 
$c$ & $\Z/3\Z$ & $[3]$ \\

& & $\{5\}$ & $3$ & $2$ & 
$a$ & $\Z/3\Z$ & $[2,7]$ \\

& & & & & 
$b$ & $\{1\}$ & $[2,2,9]$ \\

\hline

$\Q(\sqrt{-23})$ & $2$ & $\emptyset$ & $1$ & $1$ & 
$-$ & $\{1\}$ & $[2,3,7]$ \\

& &\{23\} & $1$ & $1$ & 
$-$ & $\{1\}$ & $[3,7]$ \\

\hline
\hline

$\CC_{2}$ & $2$ & $\emptyset$ & $9$ & $6$ & 
$a$ & $(\Z/3\Z)^2$ & $[2,7]$ \\
%& $\Pi$ ** \\

& & & & & 
$b$ & $\Z/3\Z$ & $[2,7,9]$ \\
%& $\Pi$* \\

& & & & & 
$c$ & $\Z/3\Z$ & $[2,9]$ \\
%& $\Pi$* \\

& & & & & 
$d$ & $\Z/3\Z$ & $[2,9]$ \\
%& $\Pi$* \\

& & & & & 
$f$ & $1$ & $[2,3,3]$ \\
%& $\Pi$ \\

& & & & & 
$g$ & $1$ & $[2,3,3]$ \\
%& $\Pi$ \\

& & $\{3\}$ & $9$ & $1$ & 
$-$ & $(\Z/3\Z)^2$ & $[7]$ \\
%& $\Pi$ **\\

\hline

$\CC_{10}$ & $2$ & $\emptyset$ & $3$ & $1$ & 
$-$ & $\Z/3\Z$ & $[2,7]$ \\

& & $\{17-\}$ & $3$ & $1$ &
$-$ & $\Z/3\Z$ & $[7]$ \\

\hline

$\CC_{18}$ & $3$ & $\emptyset$ & $9$ & $1$ &
$a$ & $(\Z/3\Z)^2$ & $[2,2,13]$ \\
%& $\Pi$ ** \\

& & $\emptyset$/\{2I\} & $1$ & $1$ & 
$b/d$ & $1$ & $[2,3,3]$ \\
%& {\color{blue} ???} \\

%&  $2I$ & $1$ & $1$ &
%$d$ & $1$ & $[2,3,3]$ \\

& & $\{2\}$ & $3$ & $3$ &
$a$ & $\Z/3\Z$ & $[2,3,13]$ \\
%& $-$ \\

&&&&&
$b$ &  $\Z/3\Z$ & $[2,3]$ \\
%& $-$\\ 

&&&&&
$c$ &  $\Z/3\Z$ & $[2,3]$ \\
%& $-$\\

\hline

$\CC_{20}$ & $2$ & $\emptyset$ & $21$ & $1$ & 
$-$ & $G_{21}$ & $[2,2,2,2,2,2]$ \\

& & $\{3-\}$ & $3$ & $2$ &  
$a$ & $\Z/3\Z$ & $[4,7]$ \\

& & & & &  
$b$ & $\{1\}$ & $[2,3,4]$ \\

& & $\{3+\}$ & $3$ & $2$ &  
$a$ & $\Z/3\Z$ & $[4,7]$ \\

& & & & &  
$b$ & $\{1\}$ & $[2,3,4]$ \\

\hline

\end{longtable}
}

% Newpage to separate the table
\newpage
%%%%%%%%%%%%%%%%%%%%%%%%%%%%%%%%%%%%%%%%%%%%%%%%%%%%%%%%%%%%%%%%%%%%%%%%%%%%%%%%

\medskip

\address{
Sergey Galkin,
National Research University Higher School of Economics\\
%and University of Vienna\\
\email{Sergey.Galkin@phystech.edu}
}

\medskip

\address{
Ludmil Katzarkov,
University of Miami and University of Vienna\\
\email{lkatzark@math.uci.edu}
}

\medskip

\address{
Anton Mellit,
International Centre for Theoretical Physics\\
%ICTP\\
\email{Mellit@gmail.com}
}

\medskip

\address{
Evgeny Shinder,
University of Edinburgh\\
\email{E.Shinder@ed.ac.uk}
%and Max-Planck-Institut f\"ur Mathematik\\
%\email{evgenyshinder2011@u.northwestern.edu}
}

\end{document}